\documentclass[12pt]{article}
\usepackage{color}
\usepackage{amsthm}
\usepackage{amsmath}
\usepackage{amssymb}
\usepackage{relsize}
\usepackage{amsfonts}
\usepackage{amsbsy}
\usepackage{textcomp} 
\usepackage[autostyle]{csquotes}

\usepackage{multicol}
\usepackage[dvips]{graphics}
\usepackage{graphics,pspicture}
\usepackage{amssymb}
\newcommand{\be}{\begin{equation}}
\newcommand{\ee}{\end{equation}}
\newcommand{\bea}{\begin{eqnarray}}
\newcommand{\eea}{\end{eqnarray}}
\newcommand{\ba}{\begin{array}}
\newcommand{\ea}{\end{array}}

\newcommand{\bc}{\begin{center}}
\newcommand{\ec}{\end{center}}
\newcommand{\ben}{\begin{enumerate}}
\newcommand{\een}{\end{enumerate}}
\newcommand{\bfi}{\begin{figure}}
\newcommand{\efi}{\end{figure}}

\newcommand{\bq}{\begin{quote}}
\newcommand{\eq}{\end{quote}}
\newcommand{\bqu}{\begin{quotation}}
\newcommand{\equ}{\end{quotation}}
\newenvironment{emphit}{\begin{itemize}}{\end{itemize}}
\newcommand{\bemp}{\begin{emphit}}
\newcommand{\eemp}{\end{emphit}}

\newcommand{\bt}{\begin{tabular}}
\newcommand{\et}{\end{tabular}}

\newtheorem{myth}{Theorem}[section]
\newtheorem{myexam}{Example}[section]
\newtheorem{mylem}{Lemma}[section]

\newtheorem{mydef}{Definition}[section]

\newtheorem{myrem}{Remark}[section]
\begin{document}
\date{}
\title{Parastrophes and Cosets of Soft Quasigroups
\footnote{2020 Mathematics Subject Classification. Primary 20N05, 03E72;
Secondary 03E75}
\thanks{{\bf Keywords and Phrases :} soft set, quasigroup, soft quasigroup, soft loop, left(right) coset, quotient of soft quasigroup, parastrophes}}
\author{Anthony Oyem\thanks{All correspondence to be addressed to this author.} \\Department of Mathematics,\\University of Lagos, \\Akoka 100213 Nigeria.\\ tonyoyem@yahoo.com\and T\`em\'it\d {\'o}p\d {\'e} Gb\d {\'o}l\'ah\`an  Jaiy\'e\d ol\'a\\
Department of Mathematics,\\ Obafemi Awolowo University, \\
Ile-Ife 220005, Nigeria.\\ tjayeola@oauife.edu.ng} 
\maketitle

\begin{abstract}
This paper introduced the concept of soft quasigroup, its parastrophes, soft nuclei, left (right) coset, distributive soft quasigroups and normal soft quasigroups. Necessary and sufficient conditions  for a soft set over a quasigroup (loop) to be a soft quasigroup (loop) were established. It was proved that a soft set over a group is a soft group if and only if it is a soft loop or either of two of its parastrophes is a soft groupoid. For a finite quasigroup, it was shown that the orders (arithmetic and geometric means) of the soft quasigroup over it and its parastrophes are equal.
It was also proved that if a soft quasigroup is distributive, then all its parastrophes are distributive, idempotent and flexible soft quasigroups.
For a distributive soft quasigroup, it was shown that its left and right cosets form families of distributive soft quasigroups that are isomorphic.
If in addition, a soft quasigroup is normal, then its left and right cosets forms families of normal soft quasigroups. On another hand, it was found that if a soft quasigroup is a normal and distributive soft quasigroup, then its left (right) quotient is a family of commutative distributive quasigroups which have a 1-1 correspondence with the left (right) coset of the soft quasigroup.
\end{abstract}

\section{Introduction}
\paragraph{}
A quasigroup is an algebraic structure which is not necessarily associative in the sense of a group.

However, there exist some interesting properties of quasigroups that make them different from group. A quasigroup may be homogeneous in nature in the sense that its group of automorphisms may be transitive. The only group with homogeneous property has just one element. Also a quasigroup has a rich outer symmetry (or duality) because any quasigroup structure is associated with six parastrophes. In general, only the transpose of a group is a group.

Effectiveness of the application of the theory of quasigroups is based on the fact that quasigroups are ``generalized permutations''. Namely, both left and right translations in quasigroups are permutations and quasigroups are characterized by this property.

The study of soft sets theory started with Molodtsov \cite{11} as a better generalization of set theory for the inquest and formal modeling of mathematical problems represented by uncertainties and vagueness due to partial and inadequate informations.

As a remedy for this defect in set theory, several advancements of set theory have been proposed and developed like the vague set theory by Gau and Buehrer \cite{51}, fuzzy set theory by Zadeh \cite{52}, rough set theory by Pawlak \cite{53}, neutrosophic set by Smarandache \cite{54}. However, all these theories have their challenges, possibly mainly due to inadequacy of the parameterization tools of the theories as pointed out in Molodtsov \cite{11}. For instance, probability theory can only deal with stochastically stable systems, where a limit of the sample mean should exit in a long series of trials.
The method of interval mathematics is not adequate for problems with different type of uncertainties, while rough set theory approach can only handle problems that involves uncertainties caused by indiscernible elements with different values in decision attributes. The fuzzy set theory approach is found most appropriate for dealing with uncertainties. It provides a tool on how to set a membership function, since the type of the membership function is defined on individual attributes.

Compared to the above mentioned mathematical tools for computing uncertainties, soft set has one important property that makes it different from other tools; parametrization. For example, it is not necessarily like membership grade in fuzzy set or approximation grade in rough set. Soft set theory has a rich potential for applications in several directions, few of which were shown by Molodtsov \cite{11} in his pioneer work. Atkas and Cagman \cite{3} did a comparison of soft sets with fuzzy sets and rough sets and showed that the both can be considered as a soft set.

Recently Oyem et al. \cite{oyem, tony} extended the results of soft sets to quasigroup by investigating the order of finite soft quasigroups and the algebraic properties of soft quasigroups.

\paragraph{}
Inspired by the study of algebraic properties of soft sets, our aim in this paper is to initiate research about soft quasigroup, its parastrophes and their cosets. It was shown that every soft quasigroup is associated to five soft quasigroups called its parastrophes. Necessary and sufficient conditions  for a soft set over a quasigroup (loop) to be a soft quasigroup (loop) were established. It was proved that a soft set over a group is a soft group if and only if it is a soft loop or either of two of its parastrophes is a soft groupoid. For a finite quasigroup $Q$, it was shown that the orders (arithmetic and geometric means) of the soft quasigroup $(F,A)_Q$ and it parastrophes are equal.

It was established that if a soft quasigroup is distributive, then all its parastrophes are distributive, idempotent and flexible soft quasigroups. For a distributive soft quasigroup $(F,A)$, it was shown that its left and right cosets form families of distributive soft quasigroups that are isomorphic. If in addition, $(F,A)$ is normal, then its left and right cosets form families of normal soft quasigroups. On another hand, it was found that if $(F,A)$ is a normal and distributive soft quasigroup, then its quotient is a family of commutative distributive quasigroups. 

\section{Preliminaries}
\paragraph{}
In this section, we review some notions and results concerning quasigroups and soft sets.

\subsection{Groupoid and Quasigroups}
\begin{mydef}(Groupoid, Quasigroup)\cite{5,6,14,2.3}
	
Let $G$ be a non-empty set. Define a binary operation ($\cdot $) on
$G$. If $x\cdot y\in G$ for all $x, y\in G$, then the pair $(G, \cdot )$ is called a \textit{groupoid} or \textit{magma}.
If each of the equations:
\begin{displaymath}
a\cdot x = b\qquad\textrm{and}\qquad y\cdot a=b
\end{displaymath}
has unique solutions in $G$ for $x$ and $y$ respectively for all $a, b \in G$, then $(G,
\cdot )$ is called a \textit{quasigroup}.
If there exists a unique element $e\in G$ called the
\textit{identity element} such that for all $x\in G$, $x\cdot
e = e\cdot x = x$, $(G, \cdot )$ is called a \textit{loop}.
\end{mydef} 
We write
$xy$ instead of $x\cdot y$, and stipulate that $\cdot$ has lower
priority than juxtaposition among factors to be multiplied. For
instance, $x\cdot yz$ stands for $x(yz)$.
Let $x$ be a fixed element in a groupoid $(G, \cdot )$. The left and right translation maps of $G$, $L_x$ and $R_x$ respectively are
defined by
\begin{displaymath}
yL_x=x\cdot y\qquad\textrm{and}\qquad yR_x=y\cdot x.
\end{displaymath}
It can now be said that a groupoid $(G, \cdot )$ is a quasigroup if
its left and right translation mappings are permutations. Since the left and right translation mappings of a
quasigroup are bijective, then the inverse mappings $L_x^{-1}$ and
$R_x^{-1}$ exist. Let
\begin{displaymath}
x\backslash y =yL_x^{-1}\qquad\textrm{and}\qquad
x/y=xR_y^{-1}
\end{displaymath}
and note that
\begin{displaymath}
x\backslash y = z\Leftrightarrow x\cdot
z=y\qquad\textrm{and}\qquad x/y = z\Leftrightarrow z\cdot y=x.
\end{displaymath}
In the language of universal algebra, a quasigroup $(G, \cdot )$ can also be represented as a quasigroup $(G, \cdot,/,\backslash )$.

For more on quasigroups, readers can check \cite{2.2,  2.5, 2.6, 2.4, 2.1, 8, 9, 1, 2,1.3,1.2}.
\begin{mydef}(Subgroupoid, Subquasigroup)\cite{5,6,14,2.3}

Let $(Q,\cdot )$ be a groupoid quasigroup and $\emptyset \ne H \subseteq Q$. Then, $H$ is called a subgroupoid (subquasigroup) of $Q$ if $(H,\cdot )$ is a groupoid (quasigroup). This is often expressed as $H \le Q$.
\end{mydef}

\begin{mydef}(Normal Subquasigroup)\cite{6,14}

Let $(Q,\cdot )$ be a quasigroup. An equivalence relation $\theta$ on a quasigroup $(Q,\cdot )$ is called a normal equivalence (or normal congruence) relation  if it satisfies the following conditions for all $a,b,c,d\in Q$:
\begin{enumerate}
\item if $ca\theta cb\Rightarrow a\theta b$;
\item if $ac\theta bc\Rightarrow a\theta b$;
\item if $a\theta b$ and $c\theta d\Rightarrow~ac\theta bd$
\end{enumerate}
Let $\emptyset \ne H \le Q$. Then, $H$ is called a normal subquasigroup of $Q$, written as $H\lhd  Q$ if $H$ is an equivalence class with respect to some normal equivalence  relation $\theta$ on $(Q,\cdot )$.
\end{mydef}

\begin{mydef}(Parastrophes of a Quasigroup \cite{14, 1, 1.1})

Let $(Q, \star)$ be a quasigroup. The five parastrophes $Q_i = (Q, \star_i),~i=1,2,3,4,5$ of $(Q, \star)$ are the quasigroups $(Q,\circledast),(Q,/),(Q,\backslash),(Q,//),(Q,\backslash\backslash)$ whose binary operations over $Q$ are defined in Table~\ref{2.1}.
\end{mydef}
%\begin{center}
\begin{table}[!hpb]
\begin{tabular}{|c||c|c|}
\hline
$\cdot $ & Parastrophic operation & Name \\
\hline
1  &$x\backslash y = z \Longleftrightarrow x \star z = y$  & left division \\
2 & $x/y = z \Longleftrightarrow z \star y = x$ & right division   \\
3 & $x \star y = z \Longleftrightarrow y \circledast x = z$ &opposite multiplication \\
4 & $x//y = z \Longleftrightarrow y/x = z \Longleftrightarrow z\star x = y$ & opposite right division\\
5 & $x \backslash \backslash y = z \Longleftrightarrow y \backslash x = z \Longleftrightarrow y \star z = x$ & opposite left division  \\
\hline
\end{tabular}\caption{Parastrophic operations on a quasigroup $(Q,\star)$}\label{2.1}
\end{table}
%\end{center}

Every quasigroup belongs to a set of six quasigroups, called adjugates by Fisher and  Yates \cite{2.1},  conjugates by Stein \cite{2.3},  parastrophes by  Sade \cite{2.2}.
Most other authors like Shchukin \cite{2.5} and, Shchukin and Gushan  \cite{2.6} and Artzy \cite{2.4} adopted the last terminology. Ogunriade et al. \cite{Ogun1}  studied a class of distributive quasigroup and their parastrophes as well as self-distributive quasigroups with key laws in Ogunriade et al. \cite{Ogun2}.

\begin{myexam}\label{Exam0}
Consider a quasigroup $(Q, \cdot)$. Its multiplication table and the multiplication tables of its five parastrophes are as displayed in Table~\ref{2.2}.
\end{myexam}
%\begin{center}
\begin{table}
\begin{tabular}{ |c | c | c | c | c | c | c |}
\hline
$\cdot$ & $ 1 $ & $ 2 $ & $ 3 $ & $ 4 $ & $ 5 $ & $ 6 $\\
\hline
1 & 1 & 2 & 3 & 4 & 6 & 5\\
2 & 2 & 1 & 5 & 6 & 3 & 4\\
3 & 3 & 5 & 4 & 1 & 2 & 6\\
4 & 4 & 6 & 1 & 3 & 5 & 2\\
5 & 5 & 4 &
 6 & 2 & 1 & 3\\
6 & 6 & 3 & 2 & 5 & 4 & 1\\
\hline
\end{tabular}
\vspace{0.5cm} \qquad 
\begin{tabular}{ |c | c | c | c | c | c | c |}
\hline
$\star$ & $ 1 $ & $ 2 $ & $ 3 $ & $ 4 $ & $ 5 $ & $ 6 $\\
\hline
1 & 1 & 2 & 3 & 4 & 5 & 6\\
2 & 2 & 1 & 5 & 6 & 4 & 3\\
3 & 3 & 5 & 4 & 1 & 6 & 2\\
4 & 4 & 6 & 1 & 3 & 2 & 5\\
5 & 6 & 3 & 2 & 5 & 1 & 4\\
6 & 5 & 4 & 6 & 2 & 3 & 1\\
\hline
\end{tabular}
\vspace{0.5cm} \qquad 
\begin{tabular}{ |c | c | c | c | c | c | c |}
\hline
$/$ & $ 1 $ & $ 2 $ & $ 3 $ & $ 4 $ & $ 5 $ & $ 6 $\\
\hline
1 & 1 & 2 & 4 & 3 & 5 & 6\\
2 & 2 & 1 & 6 & 5 & 3 & 4\\
3 & 3 & 6 & 1 & 4 & 2 & 5\\
4 & 4 & 5 & 3 & 1 & 6 & 2\\
5 & 5 & 3 & 2 & 6 & 4 & 1\\
6 & 6 & 4 & 5 & 2 & 1 & 3\\
\hline
\end{tabular}
\begin{tabular}{ |c | c | c | c | c | c | c |}
\hline
$\backslash$ & $ 1 $ & $ 2 $ & $ 3 $ & $ 4 $ & $ 5 $ & $ 6 $\\
\hline
1 & 1 & 2 & 3 & 4 & 6 & 5\\
2 & 2 & 1 & 5 & 6 & 3 & 4\\
3 & 4 & 5 & 1 & 3 & 2 & 6\\
4 & 3 & 6 & 4 & 1 & 5 & 2\\
5 & 5 & 4 & 6 & 2 & 1 & 3\\
6 & 6 & 3 & 2 & 5 & 4 & 1\\
\hline
\end{tabular}\qquad \qquad
\begin{tabular}{ |c | c | c | c | c | c | c |}
\hline
$//$ & $ 1 $ & $ 2 $ & $ 3 $ & $ 4 $ & $ 5 $ & $ 6 $\\
\hline
1 & 1 & 2 & 3 & 4 & 5 & 6\\
2 & 2 & 1 & 5 & 6 & 4 & 3\\
3 & 3 & 5 & 1 & 4 & 6 & 2\\
4 & 4 & 6 & 3 & 1 & 2 & 5\\
5 & 6 & 3 & 2 & 5 & 1 & 4\\
6 & 5 & 4 & 6 & 2 & 3 & 1\\
\hline
\end{tabular}\qquad 
\begin{tabular}{ |c | c | c | c | c | c | c |}
\hline
$\backslash \backslash$ & $ 1 $ & $ 2 $ & $ 3 $ & $ 4 $ & $ 5 $ & $ 6 $\\
\hline
1 & 1 & 2 & 3 & 4 & 5 & 6\\
2 & 2 & 1 & 6 & 5 & 3 & 4\\
3 & 4 & 6 & 1 & 3 & 2 & 5\\
4 & 3 & 5 & 4 & 1 & 6 & 2\\
5 & 5 & 3 & 2 & 6 & 4 & 1\\
6 & 6 & 4 & 5 & 2 & 1 & 3\\
\hline
\end{tabular}
\caption{Parastrophes $(Q,\star),(Q,/),(Q,\backslash),(Q,//),(Q,\backslash\backslash)$ of $(Q,\cdot)$}\label{2.2}
\end{table}
%\end{center}

%\begin{myrem}
%The binary operations of $\star,\star_i,~i=1,2,3,4,5,6$ are quasigroups and can also be defined by using the permutations $\sigma_i\in S_3,~i=1,2,3,4,5$, that is, $\sigma_0=(1),\sigma_1 = (23),\sigma_2 = (13),\sigma_3 = (132),\sigma_4 = (123), \sigma_5 = (12)$.\\
%Hence the following equations are equivalent~~$x \cdot y = z;~~ x \star z =y;~~ y / z = x;~~ y \backslash x = z;~~ z // x = y;~~ z \backslash \backslash y = x.$\\
%If $(Q, \cdot)$ is commutative, then the operations  $ \backslash = \cdot,~~ / = \star,~~ // = \backslash \backslash$; If $(Q, \cdot)$ is totally symmetric, then all the six operations are identical. 
%\end{myrem}

Parastrophes in general do not preserve the structure of a quasigroup but they do preserve all subquasigroups; all individual generators as generators and the number of generators is the same for each of the quasigroup's parastrophes.
Belousov \cite{1.5}, Dudek \cite{1.1}, Duplak \cite{1.6}, Jaiy\'e\d ol\'a \cite{1.3}, Samardziska \cite{1.2} have studied the parastrophy of quasigroups and loops in different fashions.\\

Aside the opposite parastrophe of a group, none of the other parastrophes is a group.  Jaiy\'e\d ol\'a \cite{1.3} investigated the parastrophism of associativity in quasigroups. However, the parastrophes of idempotent quasigroup is also idempotent. In some instances, the parastrophes of a given quasigroup are pairwise equal or pairwise distinct.

\begin{myth}\label{Thm1} \cite{14, 1.3}\\
Let $(Q, \star)$ be a quasigroup and $(Q, \star_i)$ be its parastrophes, $i=1,2,3,4,5$. If  $(H, \star)$ is a subquasigroup of $(Q,\star)$, then  $(H, \star_i)$ is a subquasigroup of $(Q, \star_i)$ for $i=1,2,3,4,5$.
\end{myth}

\begin{myth}(Pflugfelder \cite{14})\label{Thm2}

Let $(Q, \cdot)$ be a quasigroup (loop) and $\emptyset \ne H  \subseteq Q$. $(H, \cdot)\le (Q, \cdot)$ if and only $(H, \cdot),(H, /),(H, \backslash)$ are groupoids.
\end{myth}

\begin{myth}(Pflugfelder \cite{14})\label{Thm3}

Let $(Q, \cdot)$ be a group. Then, for any $\emptyset \ne H\subseteq Q$, the following statements are equivalent:
\begin{enumerate}
\item $(H, \cdot)$ is a subloop of $(Q, \cdot)$.
\item $(H, \cdot)$ is a subgroup of $(Q, \cdot)$.
\item $(H, /)$ is a groupoid.
\item $(H, \backslash)$ is a  groupoid.
\end{enumerate}
\end{myth}

\begin{mydef}(Distributive Quasigroup)\cite{14}

A quasigroup $(Q,\cdot)$ is called a  left distributive quasigroup if it satisfies the left distributive law $x(yz) = xy \cdot xz$. A quasigroup $(Q,\cdot)$ is called a  right distributive quasigroup if it satisfies the right distributive law
$ (yz)x = yx \cdot zx$. A quasigroup which is both a left and right distributive quasigroup is called a 
distributive quasigroup.
\end{mydef}

\begin{myth}(Pflugfelder \cite{14})\label{Thm4}

If $(Q, \cdot,/,\backslash )$ is a distributive quasigroup, then the following are true for all $x, y, z \in Q$.
\begin{enumerate}
\item $Q$ is idempotent i.e. $x^{2} = x$.
\item $L_x$ and $R_x$ are automorphisms of $(Q, \cdot)$.
\item $Q$ is flexible i.e. $x(yx)= (xy)x$.
\item $x(y \backslash z) = (xy) \backslash (xz),~ (y/z)x = (yx)/(zx),~ x \backslash (yz) = (x \backslash y)(x \backslash z),~ (yz)/x = (y/x)(z/x)$.
\end{enumerate}
\end{myth}

\begin{myth}(Pflugfelder \cite{14})\label{Thm5}

If $(Q, \cdot,/,\backslash )$ is a distributive quasigroup, then $(Q,\circledast),(Q,/),(Q,\backslash),(Q,//),(Q,\backslash\backslash)$ are distributive quasigroups.
\end{myth}

\begin{myth}(Pflugfelder \cite{14})\label{Thm6}

Let $Q$ be a distributive quasigroup such that $|Q| > 1$, then  $N_\rho (G)=\emptyset =N_\lambda (G)$.
\end{myth}

\begin{myth}(Pflugfelder \cite{14})\label{Thm7}

Let $H\le Q$ where $(Q,\cdot )$ is a distributive quasigroup. Then, all left cosets $x\cdot H$ and all right cosets $H\cdot x$ of $H$ are subquasigroups of $Q$ for any $x\in Q$.
\end{myth}

\begin{myth}(Pflugfelder \cite{14})\label{Thm8}

Let $H\le Q$ where $Q$ is a distributive quasigroup. Then, every left and right cosets of $H$ are isomorphic to $H$ and to each other. That is, $x\cdot H=xH\cong H\cong Hx=H\cdot x$ for any $x\in Q$.
\end{myth}

\begin{myth}(Pflugfelder \cite{14})\label{Thm9}

Let $H\lhd Q$ where $Q$ is a distributive quasigroup. Then, $xH,Hx\lhd Q$ for any $x\in Q$.
\end{myth}

\begin{myth}(Pflugfelder \cite{14})\label{Thm10}

Let $H\lhd Q$ where $Q$ is a distributive quasigroup. Then, $Q/H$ is a commutative distributive quasigroup.
\end{myth}

\subsection{Soft Sets and Some Operations}
\paragraph{}
We start with the notion of soft sets and operations defined on it. We refer readers to \cite{11, 3, sma2, 4,10} for earlier works on soft sets, soft groups and their operations. Throughout this subsection, $Q$ denotes an initial universe, $E$ is the set of parameters and $A \subseteq E$. 

\begin{mydef}(Soft Sets, Soft Subset, Equal Soft Sets)\cite{11,3,10}

Let $Q$ be a set and $E$ be a set of parameters. For $A\subset E$, the pair $(F, A)$ is called a soft set over $Q$ if $F(a) \subset Q$ for all $a \in A$, where $F$ is a function mapping $A$ into the set of all non-empty subsets of $Q$, i.e $F: A \longrightarrow 2^{Q}\backslash\{\emptyset\}$.
A soft set $(F, A)$ over a set $Q$ is identified or represented as a set of ordered pairs: $(F, A) = \{(a,F(a)) : a\in A~\textrm{ and}~ F(a) \in 2^Q\}$. The set
of all soft sets, over $Q$ under a set of parameters $A$, is denoted by $SS(Q_A)$.
\subparagraph{}
Let $(F, A)$ and $(H, B)$ be two soft sets over a common universe $U$, then $(H, B)$ is called a soft subset of $(F, A)$ if
\begin{enumerate}
\item $B \subseteq A$; and
\item $H(x) \subseteq F(x)$ for all $x \in B$.
\end{enumerate}
This is usually expressed as $ (F, A)\supset (H, B)$ or $(H, B)\subset (F, A)$, and $(F, A)$  is said to be a soft super set of $(H, B)$.

Two soft sets $(F, A)$ and $(H, B)$ over a common universe $U$ are said to be soft equal if $(F, A)$ is a soft subset of $(H, B)$ and $(H, B)$ is a soft subset of $(F, A)$.
\end{mydef}
\begin{mydef}(Restricted Intersection)\cite{sma2,4,10}

Let $(F, A)$ and $(G, B)$ be two soft sets over a common universe $U$ such that $A \cap B\neq\emptyset$. Then their restricted intersection is $(F, A) \cap_R (G, B) = (H, C)$ where $(H, C)$ is defined as $H(c) = F(c) \cap G(c)$ for all $c \in C$, where $C = A \cap B$.
\end{mydef} 
\begin{mydef}(Extended Intersection)\cite{sma2,4,10}

The extended intersection of two soft sets $(F, A)$ and $(G, B)$ over a common universe $U$ is the soft set $(H, C)$, where $C$ = $A \cup B$, and for all $x\in C$, $H(x)$ is defined as
\begin{displaymath}
H(x)=\left\{\begin{array}{lll}
F(x) & \mbox{if}\hspace{.2in} x\in A - B\\
G(x)  & \mbox{if}\hspace{.2in} x\in B - A\\
F(x) \cap G(x) & \mbox{if}\hspace{.2in} x\in A \cap B.
\end{array}\right.
\end{displaymath}
\end{mydef} 

\begin{mydef}(Extended Union)\cite{sma2,4,10}
	
The union of two soft sets $(F, A)$ and $(G, B)$ over $U$ is denoted by  $(F, A) \bigcup (G, B)$ and is a soft set $(H, C)$ over $U$, such that $C$  = $A \cup B$, $\forall x \in C$\\  and 
\begin{displaymath}	
H(x)=\left\{\begin{array}{lll}
F(x) & \mbox{if}\hspace{.2in} x\in A - B\\
G(x)  & \mbox{if}\hspace{.2in} x\in B - A\\
F(x) \cup G(x) & \mbox{if}\hspace{.2in} x\in A \cap B.
\end{array}\right.
\end{displaymath}
\end{mydef} 

\section{Main Results}
\subsection{Soft Quasigroups and Soft Loops}
\begin{mydef}(Soft groupoid, Soft quasigroup and Soft loop)\label{softqg}\cite{oyem, tony}

Let $(Q,\cdot)$ be a groupoid (quasigroup, loop) and $E$ be a set of parameters. For $A\subset E$, the pair $(F, A)_{(Q,\cdot)}=(F, A)_Q$ is called a soft groupoid (quasigroup, loop) over $Q$ if $F(a)$ is a subgroupoid (subquasigroup, subloop) of $Q$ for all $a \in A$, where $F: A \longrightarrow 2^{Q}\backslash\{\emptyset\}$. A soft groupoid (quasigroup, subloop) $(F, A)_{(Q,\cdot)}$ over a groupoid (quasigroup, loop) $(Q,\cdot)$ is identified or represented as a set of ordered pairs: $(F, A)_{(Q,\cdot)} = \{\left(a,(F(a),\cdot)\right) : a\in A~\textrm{ and}~ F(k) \in 2^Q\}$.
\end{mydef}
\begin{myrem}
Based on Definition \ref{softqg}, a soft quasigroup is a soft groupoid, but the converse is not necessarily true. A soft groupoid (quasigroup) is said to be finite if its underlying groupoid (quasigroup) is finite.
\end{myrem}
\begin{table}[!hpb]
\begin{center}
\begin{tabular}{|c|| c | c | c | c | c | c | c |c| }
\hline
$\cdot$ & $1$ &$2$& $3$&$4$& $5$ & $6$& $7$& $8$\\
\hline\hline
1 & 1 & 2 & 3 & 4 & 5 & 6 & 7 & 8\\
\hline
2 & 2 & 1 & 4 & 3 & 6 & 5 & 8 & 7\\
\hline
3 & 3 & 4 & 1 & 2 & 7 & 8 & 6 & 5\\
\hline
4 & 4 & 3 & 2 & 1 & 8 & 7 & 5 & 6\\
\hline
5 & 6 & 5 & 8 & 7 & 2 & 1 & 4 & 3\\
\hline
6 & 5 & 6 & 7 & 8 & 4 & 2 & 3 & l\\
\hline
7 & 8 & 7 & 5 & 6 & 3 & 4 & 1 & 2\\
\hline
8 & 7 & 8 & 6 & 5 & 4 & 3 & 2 & 1\\
\hline
\end{tabular}
\end{center}\caption{Quasigroup $(Q,\cdot)$ of order $8$}\label{Tab3.1}
\end{table}

\begin{myexam}\label{Exam1}
Let Table \ref{Tab3.1} represents the Latin square of a finite quasigroup $(Q,\cdot ),\;Q =\{1, 2, 3, 4, 5, 6, 7, 8\}$ and let $A =\{\gamma_{1}, \gamma_{2}, \gamma_{3}\}$ be any set of parameters. Let $F: A \longrightarrow 2^{Q}\backslash\{\emptyset\}$ be defined by
\begin{displaymath}
F(\gamma_{1}) = \{1, 2\},\; F(\gamma_{2}) = \{1, 2, 3, 4\},\;F(\gamma_{3}) = \{1, 2, 7, 8\}.
\end{displaymath}
Then, the pair $(F, A)$ is called a soft quasigroup over quasigroup $Q$ because each of $F(\gamma_i)\le Q,\; i = 1, 2,3.$\\
\end{myexam}

\begin{myth}\label{Thm33}
	
Let $(F, A)$ be a soft set over a group $(Q,\cdot)$. Then,  $(F, A)_{(Q,\cdot)}$ is a soft group if and only if any of the following statements is true:
	\begin{enumerate}
		\item $(F, A)_{(Q,\cdot)}$ is a soft loop.
		\item $(F, A)_{(Q,/)}$ is a soft groupoid.
		\item $(F, A)_{(Q,\backslash)}$ is a soft groupoid.
	\end{enumerate}
\end{myth}
\begin{proof}
Let $(F, A)$ be a soft set over a group $(Q,\cdot)$. $(F, A)_{(Q,\cdot)}$ is a soft group if and only if $F(a)$ is a subgroup of $Q$ for all $a\in A$. Going by Theorem \ref{Thm3}, this possible if and only any of the following statements is true:
\begin{enumerate}
	\item $F(a)$ is a subloop of $(Q, \cdot)$.
	\item $F(a)$ is a subgroupoid of $(Q,/)$.
	\item $F(a)$ is a subgroupoid of $(Q,\backslash)$.
\end{enumerate}
It can thus be concluded that $(F, A)_{(Q,\cdot)}$ is a soft group if and only if $(F, A)_{(Q,\cdot)}$ is a soft loop if and only if $(F, A)_{(Q,/)}$ is a soft groupoid if and only if $(F, A)_{(Q,\backslash)}$ is a soft groupoid. 
\end{proof}

\begin{mydef}(Soft Subquasigroup)\cite{oyem, tony}
	
Let $(F, A)$ and $(G, B)$ be two soft quasigroups over a quasigroup $Q$. Then,
\begin{enumerate}
\item $(F, A)$ is a soft subquasigroup of $(G, B)$, denoted $(F, A)\le (G, B)$, if
\begin{enumerate}
\item $A \subseteq B$, and
\item $F(a) \le G(a)$ for all $a\in A$.
\end{enumerate}
\item $(F, A)$ is soft equal to $(G, B)$, denoted $(F, A) = (G, B)$, whenever $(F,  A) \le (G, B)$  and $(G, B) \le (F, A)$.
\end{enumerate}
\end{mydef}
\begin{myexam}\label{Exam2}
Consider the Latin square Table \ref{Tab3.1} of a quasigroup $(Q,\cdot ),~Q=\{1, 2, 3, 4, 5, 6, 7, 8\}$ in Example~\ref{Exam1}. $E = \{\gamma_{1}, \gamma_{2}, \gamma_{3}, \gamma_{4}, \gamma_{5}, \gamma_{6}\}.$  Then, let $A$ = $\{\gamma_{1}, \gamma_{2}, \gamma_{3}\} \subset E$ and $B$ = $\{\gamma_{1}, \gamma_{2}, \gamma_{3},\gamma_4\} \subset E$ be two sets of parameters, where  $F(\gamma_{1}) = \{1\}, F(\gamma_{2}) = \{1, 2\}, F(\gamma_{3}) = \{1, 2, 7, 8\}$ and $G(\gamma_1) = \{1, 2\}$, $G(\gamma_2) = \{1, 2, 3, 4\}$, $G(\gamma_3) = Q$.
Then, $(F, A) \le (G, B)$;
since $A \subseteq B$ and $F(\gamma) \le G(\gamma)$, for all $\gamma \in A$. But $(G, B) \not\le (F, A)$. Hence $(F, A) \neq (G,B)$.

This example shows that two soft quasigroups can not be equal if they have different set parameters.
\end{myexam}

\subsection{Parastrophes of Soft Quasigroups}
\begin{mydef}(Parastrophes of Soft Quasigroup)

Let $(F, A)$ be a soft quasigroup over a quasigroup $(Q,\star)$. The five parastrophes of $(F, A)_{(Q,\star)}$ are the soft sets $(F, A)_{(Q,\star_i)}~i=1,2,3,4,5$.
\end{mydef}

\begin{mylem}\label{Lem1}
Let $(F, A)$ be a soft set over a quasigroup $(Q, \cdot, \backslash, /)$, then the following statements are equivalent.
\begin{multicols}{2}
\begin{enumerate}
\item $(F, A)_{(Q, \cdot)}$ is a soft quasigroup.
\item $(F, A)_{(Q, \circledast)}$ is a soft quasigroup.
\item $(F, A)_{(Q, \backslash)}$ is a soft quasigroup .
\item $(F, A)_{(Q, /)}$ is a soft quasigroup.
\item $(F, A)_{(Q, \backslash \backslash)}$ is a soft quasigroup
\item $(F, A)_{(Q, //)}$ is a soft quasigroup
\end{enumerate}
\end{multicols}
\end{mylem}
\begin{proof}
Let $(F, A)_{(Q, \cdot)}$ be a soft quasigroup. $(Q, \cdot)$ is a quasigroup if and only if $(Q, \circledast)$ is a quasigroup if and only if $(Q, \backslash)$ is a quasigroup if and only if $(Q, /)$ is a quasigroup if and only if $(Q, //)$ is a quasigroup, if and only if   $(Q, \backslash \backslash)$ is a quasigroup. Thus, for any $F(a) \subset Q,~a\in A$ and based on Theorem~\ref{Thm1}, the following are equivalent: 
\begin{itemize}
\item $(F(a), \cdot)$ is a subquasigroup of quasigroup $(Q, \cdot)$ for $a\in A$.
\item $(F(a), \circledast)$ is a subquasigroup of quasigroup $(Q, \circledast)$ for $a\in A$.
\item $(F(a), \backslash)$ is a subquasigroup of quasigroup $(Q, \backslash)$ for $a\in A$.
\item $(F(a), \backslash\backslash)$ is a subquasigroup of quasigroup $(Q, \backslash\backslash)$ for $a\in A$.	
\item $(F(a), /)$ is a subquasigroup of quasigroup $(Q,/)$ for $a\in A$.
\item $(F(a), //)$ is a subquasigroup of quasigroup $(Q,//)$ for $a\in A$.
\end{itemize}
Therefore, $(F, A)_{(Q, \star_i)}$ is a soft quasigroup if and only if $(F, A)_{(Q, \star_j)}$ is a soft quasigroup for any $i,j=0,1,2,3,4,5$.
\end{proof}

\begin{myrem}
Based on Lemma~\ref{Lem1}, the soft quasigroup  $(F, A)_{(Q, \cdot)}$ and its five parastrophes have the same set of parameters but need not be equal.
\end{myrem}

\begin{myth}
Let $(Q, \cdot)$ be a quasigroup (loop). $(F, A)_{(Q, \cdot)}$ is a soft quasigroup (loop) if and only the following are true: 
\begin{multicols}{2}
\begin{enumerate}
\item $(F, A)_{(Q, \cdot)}$ is a soft groupoid;
\item $(F, A)_{(Q, \backslash)}$ is a soft groupoid;
\item $(F, A)_{(Q, /)}$ is a soft groupoid;
\end{enumerate}
\end{multicols}
\end{myth}
\begin{proof}
Given that $(F, A)_{(Q, \cdot)}$ is a soft quasigroup, then $(F(a), \cdot)\le (Q, \cdot)$ for all $a\in A$. By Theorem~\ref{Thm2}, this is possible if and only $(F(a), \cdot),(F(a), /)$ and $(F(a), \backslash)$ are subgroupoids of $(Q, \cdot),(Q, /)$ and $(Q, \backslash)$ respectively. This last statement is true if and only if $(F, A)_{(Q, \cdot)},(F, A)_{(Q, /)}$ and $(F, A)_{(Q, \backslash)}$ are soft groupoids. The proof for when $(F, A)_{(Q, \cdot)}$ is a soft loop is similar.	
\end{proof}

\begin{myexam}\label{Exam3}
Consider the Latin squares in Table~\ref{2.2} of a quasigroup $(Q,\cdot)$ and its parastrophes. 
Let $A =\{\gamma_{1}, \gamma_{2}, \gamma_{3}\}$ be any set of parameters. Let $F: A \longrightarrow 2^{Q}\backslash\{\emptyset\}$ be defined by
\begin{displaymath}
F(\gamma_{1}) = \{1\},\; F(\gamma_{2}) = \{1, 2\},\;F(\gamma_{3}) = \{1, 3,4\}.
\end{displaymath}
Then, by Lemma~\ref{Lem1}, the soft set $(F, A)_{(Q,\cdot)}$ and its parastrophes are soft quasigroups with subquasigroups $F(\gamma_i)\le Q,\; i = 1, 2, 3$ which are represented by the Latin squares in Table~\ref{Tab3.3}.\\
\begin{table}[!hpb]
\begin{center}
\begin{tabular}{|c||c|c|}
\hline
$\cdot $ & 1 \\
\hline\hline
1 & 1  \\
\hline
\end{tabular}  $\equiv \left(F(\gamma_1),\cdot\right)$ \qquad \qquad\begin{tabular}{|c||c|c|}
\hline
$\cdot $ & 1 & 2 \\
\hline\hline
1 & 1 & 2  \\
2 & 2 & 1 \\
\hline
\end{tabular} $\equiv \left(F(\gamma_2),\cdot\right)$ \qquad \qquad \begin{tabular}{|c||c|c|c|c|}
\hline
$\cdot $ & 1 & 3 & 4\\
\hline\hline
1 & 1 & 3 & 4\\
3 & 3 & 4 & 1 \\
4 & 4 & 1 & 3 \\
\hline
\end{tabular} $\equiv \left(F(\gamma_3),\cdot \right)$
\end{center}
%\end{table}
%\begin{table}[!hpb]
\begin{center}
\begin{tabular}{|c||c|c|}
\hline
$\circledast $ & 1 \\
\hline\hline
1 & 1  \\
\hline
\end{tabular}  $\equiv \left(F(\gamma_1),\circledast\right)$ \qquad \qquad\begin{tabular}{|c||c|c|}
\hline
$\circledast $ & 1 & 2 \\
\hline\hline
1 & 1 & 2  \\
2 & 2 & 1 \\
\hline
\end{tabular} $\equiv \left(F(\gamma_2),\circledast\right)$ \qquad  \begin{tabular}{|c||c|c|c|c|}
\hline
$\circledast $ & 1 & 3 & 4\\
\hline\hline
1 & 1 & 3 & 4\\
3 & 3 & 4 & 1 \\
4 & 4 & 1 & 3 \\
\hline
\end{tabular} $\equiv \left(F(\gamma_3),\circledast\right)$
\end{center}
%\end{table}
%\begin{table}[!hpb]
\begin{center}
\begin{tabular}{|c||c|c|}
\hline
$/$ & 1 \\
\hline\hline
1 & 1  \\
\hline
\end{tabular}  $\equiv \left(F(\gamma_1),/\right)$ \qquad \qquad\begin{tabular}{|c||c|c|}
\hline
$/$ & 1 & 2 \\
\hline\hline
1 & 1 & 2  \\
2 & 2 & 1 \\
\hline
\end{tabular} $\equiv \left(F(\gamma_2),/\right)$ \qquad \qquad \begin{tabular}{|c||c|c|c|c|}
\hline
$/ $ & 1 & 3 & 4\\
\hline\hline
1 & 1 & 4 & 3\\
3 & 3 & 1 & 4 \\
4 & 4 & 3 & 1 \\
\hline
\end{tabular} $\equiv \left(F(\gamma_3),/\right)$
\end{center}
%\end{table}
%\begin{table}[!hpb]
\begin{center}
\begin{tabular}{|c||c|c|}
\hline
$\backslash $ & 1 \\
\hline\hline
1 & 1  \\
\hline
\end{tabular}  $\equiv \left(F(\gamma_1),\backslash\right)$ \qquad \qquad\begin{tabular}{|c||c|c|}
\hline
$\backslash $ & 1 & 2 \\
\hline\hline
1 & 1 & 2  \\
2 & 2 & 1 \\
\hline
\end{tabular} $\equiv \left(F(\gamma_2),\backslash\right)$ \qquad \qquad \begin{tabular}{|c||c|c|c|c|}
\hline
$\backslash $ & 1 & 3 & 4\\
\hline\hline
1 & 1 & 3 & 4\\
3 & 4 & 1 & 3 \\
4 & 3 & 4 & 1 \\
\hline
\end{tabular} $\equiv \left(F(\gamma_3),\backslash\right)$
\end{center}
%\end{table}
%\begin{table}[!hpb]
\begin{center}
\begin{tabular}{|c||c|c|}
\hline
$// $ & 1 \\
\hline\hline
1 & 1  \\
\hline
\end{tabular}  $\equiv \left(F(\gamma_1),//\right)$ \qquad \begin{tabular}{|c||c|c|}
\hline
$//$ & 1 & 2 \\
\hline\hline
1 & 1 & 2  \\
2 & 2 & 1 \\
\hline
\end{tabular} $\equiv \left(F(\gamma_2),//\right)$ \qquad  \begin{tabular}{|c||c|c|c|c|}
\hline
$// $ & 1 & 3 & 4\\
\hline\hline
1 & 1 & 3 & 4\\
3 & 3 & 1 & 4 \\
4 & 4 & 3 & 1 \\
\hline
\end{tabular} $\equiv \left(F(\gamma_3),//\right)$
\end{center}
%\end{table}
%\begin{table}[!hpb]
\begin{center}
\begin{tabular}{|c||c|c|}
\hline
$\backslash \backslash $ & 1 \\
\hline\hline
1 & 1  \\
\hline
\end{tabular}  $\equiv \left(F(\gamma_1),\backslash \backslash\right)$ \qquad \begin{tabular}{|c||c|c|}
\hline
$\backslash \backslash $ & 1 & 2 \\
\hline\hline
1 & 1 & 2  \\
2 & 2 & 1 \\
\hline
\end{tabular} $\equiv \left(F(\gamma_2),\backslash \backslash\right)$ \qquad \begin{tabular}{|c||c|c|c|c|}
\hline
$\backslash \backslash $ & 1 & 3 & 4\\
\hline\hline
1 & 1 & 3 & 4\\
3 & 4 & 1 & 3 \\
4 & 3 & 4 & 1 \\
\hline
\end{tabular} $\equiv \left(F(\gamma_3),\backslash \backslash\right)$
\end{center}\caption{Parastrophes of Soft quasigroup $(F, A)_{(Q,\cdot )}$}\label{Tab3.3}
\end{table}
\end{myexam}

\begin{mydef}(Distributive Soft Quasigroup)

A soft quasigroup $(F, A)_Q$ is called a distributive soft quasigroup if $Q$ is a distributive quasigroup.
\end{mydef}
\begin{myth}
Let $(F, A)_{(Q,\cdot)}$ be a distributive soft quasigroup. Then
\begin{enumerate}
\item  its five parastrophes $(F, A)_{(Q, \circledast)},(F, A)_{(Q, \backslash)},(F, A)_{(Q, /)},(F, A)_{(Q, \backslash \backslash)},(F, A)_{(Q, //)}$ are distributive soft quasigroups.
\item $(F, A)_{(Q,\cdot)},(F, A)_{(Q, \circledast)},(F, A)_{(Q, \backslash)},(F, A)_{(Q, /)},(F, A)_{(Q, \backslash \backslash)},(F, A)_{(Q, //)}$ are idempotent soft quasigroups.
\item $(F, A)_{(Q,\cdot)},(F, A)_{(Q, \circledast)},(F, A)_{(Q, \backslash)},(F, A)_{(Q, /)},(F, A)_{(Q, \backslash \backslash)},(F, A)_{(Q, //)}$ are flexible soft quasigroups.
\end{enumerate}
\end{myth}
\begin{proof}
Let $(F, A)_{(Q,\cdot)}$ be a soft quasigroup. Then, by Lemma~\ref{Lem1}, $(F, A)_{(Q, \circledast)},(F, A)_{(Q, \backslash)},(F, A)_{(Q, /)},(F, A)_{(Q, \backslash \backslash)},(F, A)_{(Q, //)}$ are soft quasigroups.

$(F, A)_{(Q,\cdot)}$ is a distributive soft quasigroup means that $(Q,\cdot)$ is a distributive quasigroup. Thus, by Theorem~\ref{Thm5}, $(Q,\circledast),(Q,/),(Q,\backslash),(Q,//),(Q,\backslash\backslash)$ are distributive quasigroups. Consequently, $(F, A)_{(Q, \circledast)},(F, A)_{(Q, \backslash)},(F, A)_{(Q, /)},(F, A)_{(Q, \backslash \backslash)},(F, A)_{(Q, //)}$ are distributive soft quasigroups.	

Following Theorem~\ref{Thm4}(1,3), $(F, A)_{(Q,\cdot)},(F, A)_{(Q, \circledast)},(F, A)_{(Q, \backslash)},(F, A)_{(Q, /)},(F, A)_{(Q, \backslash \backslash)},$
$(F, A)_{(Q, //)}$ are flexible soft quasigroups and idempotent soft quasigroups
\end{proof}

\begin{mylem}
Let $(F, A)_Q$ be a distributive soft quasigroup, such that $|Q| > 1$, then $(F, A)$ is neither a left nor right nuclear.
\end{mylem}
\begin{proof}
If $(F, A)_Q$ is a distributive soft quasigroup, then $Q$ is a distributive quasigroup. Going by Theorem~\ref{Thm6}, since $|Q| > 1$, then, $N_{\rho}(Q)=\emptyset =N_{\lambda}(Q)$. So, there does not exist $a\in A$ such that $F(a)=N_{\rho}(Q)$ or $F(a)=N_{\lambda}(Q)$. Therefore, $(F, A)$ is neither a left nor right nuclear.
\end{proof}

\begin{mydef}
Let $(F, A)_Q$ be a soft quasigroup. For a fixed $x\in Q$, let $F_x,{}_xF: A \longrightarrow 2^{Q}\backslash{\emptyset}$ such that $F_x(a) = F(a)\cdot x$ and ${}_xF(a) = x\cdot F(a)$ for all $a\in A$.
\begin{enumerate}
\item The soft set $(F_x, A)$ will be called a $x$-right coset soft set of $(F, A)$ and at times denoted by $\left(F_x^\rho, A\right)$. The family $\left\{\left(F_x^\rho, A\right)_Q\right\}_{x\in Q}$ of soft sets will be represented by $\left(F_Q^\rho, A\right)$ and called the right coset of $(F, A)_Q$.
\item The soft set $\left({}_xF, A\right)$ will be called a $x$-left coset soft set of $(F, A)$ and at times denoted by $\left({}_xF^\lambda, A\right)$. The family $\left\{\left({}_xF^\lambda, A\right)_Q\right\}_{x\in Q}$ of soft sets will be represented by $\left(F_Q^\lambda, A\right)$ and called the left coset of $(F, A)_Q$.
\end{enumerate}
\end{mydef}

\begin{mylem}\label{Lem2}
Let $(F, A)_Q$ be a distributive soft quasigroup. Then, $\left(F_Q^\rho, A\right)=\left\{\left(F_x^\rho, A\right)_Q\right\}_{x\in Q}$ and $\left(F_Q^\lambda, A\right)=\left\{\left({}_xF^\lambda, A\right)_Q\right\}_{x\in Q}$  are both families of distributive soft quasigroups.
\end{mylem}
\begin{proof}
If $(F, A)_{(Q,\cdot )}$ is a distributive soft quasigroup, then $(Q,\cdot )$ is a distributive quasigroup and for all $a\in A,~F(a)\le Q$. Going by Theorem~\ref{Thm7}, $x\cdot F(a)\le Q$ for any fixed $x\in Q$. Thus, $(F, A)_{(Q,\cdot )}$ is a distributive soft quasigroup for any fixed $x\in Q$ and consequently, $\left(F_Q^\lambda, A\right)=\left\{\left({}_xF^\lambda, A\right)_Q\right\}_{x\in Q}$ is a family of distributive soft quasigroups. A similar argument goes for $\left(F_Q^\rho, A\right)=\left\{\left(F_x^\rho, A\right)_Q\right\}_{x\in Q}$.
\end{proof}

\begin{mydef}
Let $(F, A)_Q$ and $(G, A)_Q$ be soft quasigroup over $Q$, then $(F, A)_Q \cong (G, A)_Q$ if $F(a) \cong G(a)$ for all $a \in A$. 
\end{mydef}
\begin{mydef}
Let $(F, A)_Q$ be a soft quasigroup over $Q$.  $(F, A)_Q$ is called a normal soft quasigroup if $F(a)\lhd Q$ for all $a\in A$.
\end{mydef}

\begin{myth}\label{Thm11}
Let $(F, A)_Q$ be a distributive soft quasigroup. Then:
\begin{enumerate}
\item $(F, A)_Q \cong \left({}_xF^\lambda, A\right)_Q \cong \left(F_x^\rho, A\right)_Q$ for any fixed $x\in Q$. Furthermore, $\left(F_Q^\rho, A\right)\cong \left(F_Q^\lambda, A\right)$.
\item If $(F, A)_Q$ is a normal soft quasigroup, then $\left(F_Q^\rho, A\right)=\left\{\left(F_x^\rho, A\right)_Q\right\}_{x\in Q}$ and $\left(F_Q^\lambda, A\right)=\left\{\left({}_xF^\lambda, A\right)_Q\right\}_{x\in Q}$  are isomorphic families of normal soft quasigroups.
\end{enumerate}
\end{myth}
\begin{proof}
\begin{enumerate}
\item If $(F, A)_{(Q,\cdot)}$ is a distributive soft quasigroup, then, $(Q,\cdot)$ is a distributive quasigroup. 

Going by Lemma~\ref{Lem2}, $\left(F_Q^\rho, A\right)=\left\{\left(F_x^\rho, A\right)_Q\right\}_{x\in Q}$ and $\left(F_Q^\lambda, A\right)=\left\{\left({}_xF^\lambda, A\right)_Q\right\}_{x\in Q}$  are both families of distributive soft quasigroups. Using Theorem~\ref{Thm8}, $F(a)\cong F_x(a)$ and $F(a)\cong {}_xF(a) $ for all $a\in A$ and for any fixed $x\in Q$. So, $(F, A)_Q \cong \left({}_xF^\lambda, A\right)_Q \cong \left(F_x^\rho, A\right)_Q$ for any fixed $x\in Q$. More so, $\left(F_Q^\rho, A\right)\cong \left(F_Q^\lambda, A\right)$.
\item 	If $(F, A)_Q$ is a normal soft quasigroup, then, $F(a)\lhd Q$ for all $a\in A$. 

Using Theorem~\ref{Thm9}, $F_x(a)\lhd Q$ and ${}_xF(a)\lhd Q$ for all $a\in A$ and for any fixed $x\in Q$. So, $\left({}_xF^\lambda, A\right)_Q$ and  $\left(F_x^\rho, A\right)_Q$ are normal soft quasigroups for any fixed $x\in Q$.

Thus, $\left(F_Q^\rho, A\right)=\left\{\left(F_x^\rho, A\right)_Q\right\}_{x\in Q}$ and $\left(F_Q^\lambda, A\right)=\left\{\left({}_xF^\lambda, A\right)_Q\right\}_{x\in Q}$  are isomorphic families of normal soft quasigroups.
\end{enumerate}
\end{proof}

\begin{mydef}
Let $(F, A)_Q$ be a soft quasigroup, then the family $\left\{Q/F(a)\right\}_{a \in A}$ will be called the left quotient of $(F, A)_Q$ in $Q$ while $\left\{Q\backslash F(a)\right\}_{a \in A}$ will be called the right quotient of $(F, A)_Q$ in $Q$.
\end{mydef}
\begin{myth}
Let $(F, A)_Q$ be a normal and distributive soft quasigroup. Then:
\begin{enumerate}
\item the left quotient of $(F, A)_Q$ in $Q$ is a family of commutative distributive quasigroups and has a 1-1 correspondence with $\left(F_Q^\lambda, A\right)$.
\item the right quotient of $(F, A)_Q$ in $Q$ is a family of commutative distributive quasigroups and has a 1-1 correspondence with $\left(F_Q^\rho, A\right)$.
\end{enumerate}
\end{myth}
\begin{proof}
\begin{enumerate}
\item If $(F, A)_Q$ is a normal and distributive soft quasigroup, then, $F(a)\lhd Q$ for all $a\in A$. Going by Theorem \ref{Thm10}, $Q/F(a)$ is a commutative distributive quasigroup for each $a\in A$. Thus, the left quotient of $(F, A)_Q$ i.e. $\left\{Q/F(a)\right\}_{a \in A}$  is a family of commutative distributive quasigroups.
\item Similar argument.
\end{enumerate}
\end{proof}

\begin{mydef}(Order of Soft Quasigroup)(\cite{oyem})
	
	Let $(F, A)$ be a soft quasigroup over a finite quasigroup $(Q,\cdot )$. The order of the soft quasigroup $(F, A)$ or $(F, A)_{(Q,\cdot )}$ will be defined as
	\begin{displaymath}
		|(F, A)|_{(Q,\cdot )} = |(F, A)| = \sum\limits_{a \in A}|F(a)|, ~~for~~ F(a) \in (F, A)~~ and ~~a \in A.
	\end{displaymath}
	where the sum is over distinct proper subquasigroups $F(a) \in (F, A),~a \in A$.
\end{mydef}
\begin{mydef}(Arithmetic and Geometric Means of Finite Soft Quasigroup)(\cite{oyem})
	
	Let $(F, A)$ be a soft quasigroup over a finite quasigroup  $(Q,\cdot )$. The arithmetic mean and geometric mean of $(F, A)_{(Q,\cdot )}$ will be defined respectively as
	\begin{displaymath}
		\displaystyle \mathcal{AM}(F, A)_{(Q,\cdot )} = \frac{1}{|A|}\sum\limits_{a \in A}|F(a)|\quad\textrm{and}\qquad \mathcal{GM}(F, A)_{(Q,\cdot )}=\sqrt[|A|]{\prod\limits_{a \in A}|F(a)|}
	\end{displaymath}
\end{mydef}
\begin{myth}\label{oyemoyem}
Let $(F, A)$ be a soft quasigroup over a quasigroup $(Q,\star)$. Then 
\begin{enumerate}
	\item $|(F, A)|_{(Q,\star)}=|(F, A)|_{(Q,\star_i)}$ for each $~i=1,2,3,4,5$.
	\item $\mathcal{AM}(F, A)_{(Q,\star )} =\mathcal{AM}(F, A)_{(Q,\star_i)}$ for each $~i=1,2,3,4,5$.
	\item $\mathcal{GM}(F, A)_{(Q,\star )} =\mathcal{GM}(F, A)_{(Q,\star_i)}$ for each $~i=1,2,3,4,5$.
\end{enumerate}
\end{myth}
\begin{proof}
Let $(F, A)_{(Q,\star)}$ be a soft quasigroup, then by Lemma \ref{Lem1}, $(F, A)_{(Q,\star_i)}$ is a soft quasigroup for each $~i=1,2,3,4,5$. Thus,
\begin{enumerate}
	\item $|(F, A)|_{(Q,\star)}=\sum\limits_{a \in A}|F(a)|=|(F, A)|_{(Q,\star_i)}$ for each $~i=1,2,3,4,5$.
	\item $\mathcal{AM}(F, A)_{(Q,\star )} =\frac{1}{|A|}\sum\limits_{a \in A}|F(a)|=\mathcal{AM}(F, A)_{(Q,\star_i)}$ for each $~i=1,2,3,4,5$.
	\item $\mathcal{GM}(F, A)_{(Q,\star )} =\sqrt[|A|]{\prod\limits_{a \in A}|F(a)|}=\mathcal{GM}(F, A)_{(Q,\star_i)}$ for each $~i=1,2,3,4,5$.
\end{enumerate}
Thus, for any finite quasigroup, the orders (arithmetic and geometric means) of the soft quasigroup over it and its parastrophes are equal.
\end{proof}

\begin{myrem}
Based on Theorem \ref{oyemoyem}, the Maclaurin’s inequality and several other inequalities for a finite soft quasigroup obtained in \cite{oyem} are true in all the five parastrophes of any finite soft quasigroup.
\end{myrem}

\section{Conclusion and Further Studies}
\paragraph{}
We considered soft set over non-associative algebraic structures groupoid, quasigroup and loop, which was motivated by the study of algebraic structures in the context soft sets, neutrosophic sets and rough sets. In this paper, we introduced and studied parastrophes of soft quasigroup, soft nuclei, and distributive soft quasigroups. It was cogently shown that any soft quasigroup belongs to a particular family of six soft quasigroups (which are not necessarily soft equal). A 1-1 correspondence  was found between the left (right) quotient of a soft quasigroup and the left (right) coset of the soft quasigroup. Some examples were given for illustration. On the basis of these results, soft quasigroup theory can be explored and their varieties studied. Oyem et al. \cite{Oyem10} just studied soft neutrosophic quasigroups.

\end{document}